\documentclass[a4paper,11pt]{article}

\usepackage[a4paper, margin={1in, 1in}]{geometry}
\usepackage{mathtools}
\usepackage{algorithmic}
\usepackage{algorithm}
\usepackage{amsthm}
\usepackage{amssymb}
\usepackage{amsmath}
\usepackage{enumitem}
\usepackage{float}
\usepackage{etoolbox}
\patchcmd{\thmhead}{(#3)}{#3}{}{}

\theoremstyle{plain}
\newtheorem{theorem}{Theorem}[section]

\newtheorem{open}[theorem]{Problem}

\newtheorem{prop}[theorem]{Proposition}
\newtheorem{lemma}[theorem]{Lemma}
\newtheorem{cor}[theorem]{Corollary}
\newtheorem{obs}[theorem]{Observation}

\theoremstyle{definition}

\theoremstyle{remark}

\newtheorem*{remark*}{Remark}

\begin{document}

\title{Efficient Removal Lemmas for Matrices}

\author{
Noga Alon\thanks{Sackler School of Mathematics
	and Blavatnik School of
	Computer Science, Tel Aviv University, Tel Aviv 69978, Israel.
	Email: {\tt nogaa@tau.ac.il}.  Research supported in part by a
	USA-Israeli
	BSF grant 2012/107, by an ISF grant 620/13 and
	by the Israeli I-Core program.} \and
Omri Ben-Eliezer\thanks{Blavatnik School of
Computer Science, Tel Aviv University, Tel Aviv 69978, Israel.
Email: {\tt omrib@mail.tau.ac.il.}}}

\maketitle
\begin{abstract}
The authors and Fischer recently proved that any hereditary property of two-dimensional
matrices (where the row and column order is not ignored) over a finite alphabet
 is testable with a constant number of queries, by establishing the following
(ordered) matrix removal lemma: For any finite alphabet $\Sigma$, any hereditary property
$\mathcal{P}$ of matrices over $\Sigma$, and any $\epsilon > 0$,
there exists $f_{\mathcal{P}}(\epsilon)$ such that for any matrix $M$ over $\Sigma$ that is $\epsilon$-far from satisfying $\mathcal{P}$, most of the $f_{\mathcal{P}}(\epsilon) \times f_{\mathcal{P}}(\epsilon)$ submatrices of $M$ do not satisfy $\mathcal{P}$.
Here being $\epsilon$-far from $\mathcal{P}$ means that one needs to modify at least an $\epsilon$-fraction of the entries of $M$ to make it satisfy $\mathcal{P}$.

However, in the above general removal lemma, $f_{\mathcal{P}}(\epsilon)$ grows very fast as a function of $\epsilon^{-1}$, even when $\mathcal{P}$ is characterized by a single forbidden submatrix.
In this work we establish much more efficient removal lemmas for several special cases of the above problem.
In particular, we show the following: For any fixed $s \times t$ binary matrix $A$ 
and any $\epsilon > 0$ there exists $\delta > 0$ polynomial in $\epsilon$, such that 
for any binary matrix $M$ in which less than a $\delta$-fraction of the $s \times t$ submatrices
are equal to $A$, there exists a set of less than an $\epsilon$-fraction of the entries of $M$ that intersects every $A$-copy in $M$.

We generalize the work of Alon, Fischer and Newman
[SICOMP'07] and make progress towards proving one of their conjectures.
The proofs combine their efficient conditional regularity lemma for matrices
 with additional combinatorial and probabilistic
ideas.
\end{abstract} 

\section{Introduction}
\label{sec:intro}
\emph{Removal lemmas} are structural combinatorial results that relate the \emph{density} of ``forbidden'' substructures in
a given large structure $S$ with the \emph{distance} of $S$ from not containing any of the forbidden substructures, stating that if $S$ contains a small number of forbidden substructures, then one can make $S$ free of such substructures by making only a small number of modifications in it. 
Removal lemmas are closely related to many problems in Extremal Combinatorics, and have direct implications in Property Testing and other areas of Mathematics and Computer Science, such as Number Theory and Discrete Geometry.

The first known removal lemma has been the celebrated (non-induced) \emph{graph removal lemma}, established by Rusza and Szemer\'edi
\cite{RuszaSzemeredi1976} (see also
\cite{AlonDuke1994,AlonFischerKrivelevichSzegedy2000}). 
This fundamental result in Graph Theory states that for any fixed graph $H$
on $h$ vertices and any $\epsilon > 0$ there exists $\delta > 0$, such that for any graph $G$ on $n$ vertices that contains at least $\epsilon n^2$ copies of $H$ that are pairwise edge-disjoint, the total number of $H$-copies in $G$ is at least $\delta n^h$. 
Many extensions and strengthenings of the graph removal lemma have been obtained, as is described in more detail in Section \ref{sec:previous}.

In this work, we consider removal lemmas for two-dimensional \emph{matrices} (with row and column order) over a finite alphabet. For simplicity, the results are generally stated for square matrices, but are easily generalizable to non-square matrices.
Some of the results also hold for matrices in more than two dimensions.
The notation below is given for two-dimensional matrices, but carries over naturally to other combinatorial structures, such as graphs and multi-dimensional matrices.

An $m \times n$ matrix $M$ over the alphabet $\Gamma$ is viewed here as a function $M:[m] \times [n] \to \Gamma$, and the row and column order is dictated by the natural order on their indices. 
Any matrix that can be obtained from a matrix $M$ by deleting some of its rows and columns (while preserving the row and column order) is considered a \emph{submatrix} of $M$. We say that $M$ is \emph{binary} if the alphabet is $\Gamma = \{0,1\}$ and \emph{ternary} if $\Gamma = \{0,1,2\}$.
A \emph{matrix property} $\mathcal{P}$ over $\Gamma$ is simply a collection of matrices $M:[m] \times [n] \to \Gamma$. 
A matrix is \emph{$\epsilon$-far} from $\mathcal{P}$ if one
needs to change at least an $\epsilon$-fraction of its entries to get a matrix
that satisfies $\mathcal{P}$.
A property $\mathcal{P}$ is \emph{hereditary} if it is closed under taking submatrices, that is, if $M \in \mathcal{P}$ then any submatrix $M'$ of $M$ satisfies $M' \in \mathcal{P}$. 
For any family $\mathcal{F}$ of matrices over $\Gamma$, the property of $\mathcal{F}$-freeness, 
denoted by $\mathcal{P}_{\mathcal{F}}$, consists of all matrices over $\Gamma$ that do not contain a submatrix from $\mathcal{F}$.
Observe that $\mathcal{P}$ is hereditary if and only if it is characterized by some family $\mathcal{F}$ of forbidden submatrices, i.e.\@ $\mathcal{P} = \mathcal{P}_{\mathcal{F}}$.

While the investigation of graph removal lemmas has been quite extensive, as described in Section \ref{sec:previous} below, the first known removal lemma for ordered graph-like two-dimensional structures, and specifically for (row and column ordered) matrices, was only obtained very recently by the authors and Fischer \cite{AlonBenEliezerFischer2017}.
\begin{theorem}[\cite{AlonBenEliezerFischer2017}]
	\label{thm:ABF_matrix_removal_lemma}
	Fix a finite alphabet $\Gamma$. For any hereditary property $\mathcal{P}$ of matrices over $\Gamma$ and 
	any $\epsilon > 0$ there exists $f_{\mathcal{P}}(\epsilon)$ satisfying the following.
	If a matrix $M$ is $\epsilon$-far from $\mathcal{P}$ then at least a $2/3$-fraction of the
	$f_{\mathcal{{P}}}(\epsilon) \times f_{\mathcal{{P}}}(\epsilon)$ submatrices of $M$
	do not satisfy $\mathcal{P}$.
\end{theorem}
However, even when $\mathcal{P}$ is characterized by a single forbidden submatrix, the upper bound 
on $f_{\mathcal{P}}(\epsilon)$ guaranteed by the removal lemma in \cite{AlonBenEliezerFischer2017} is very large; in fact, it is at least as large as a wowzer (tower of towers) type function of $\epsilon$. 
On the other hand, a lower bound of Fischer and Rozenberg \cite{FischerRozenberg2007} implies that one cannot hope for a polynomial dependence of $f_{\mathcal{P}}(\epsilon)$ in $\epsilon^{-1}$ in general (for the non-binary case), even when $\mathcal{P}$ is characterized by a single forbidden submatrix.

Thus, it is natural to ask for which hereditary matrix properties $\mathcal{P}$ there exist removal lemmas with more reasonable upper bounds on $f_{\mathcal{P}}(\epsilon)$, and specifically, to identify large families of properties $\mathcal{P}$ for which $f_{\mathcal{P}}(\epsilon)$ is \emph{polynomial} in $\epsilon^{-1}$.
In this work we focus on this question, mainly for matrices over a binary alphabet.

A natural motivation for the investigation of removal lemmas comes from \emph{property testing}. 
This active field of study in computer science, initiated by Rubinfeld and Sudan
\cite{RubinfeldSudan1996} (see \cite{GoldreichGoldwasserRon1998} for the graph
case), is dedicated to finding fast algorithms to distinguish between objects
that satisfy a certain property and objects that are far from satisfying this
property; these algorithms are called testers. 
An \emph{$\epsilon$-tester} for a matrix property $\mathcal{P}$
is a (probabilistic) algorithm that is given query access to the entries of the 
input matrix $M$, and is required to distinguish, with error probability at most $1/3$, 
between the case that $M$ satisfies $\mathcal{P}$ and the case that $M$ is $\epsilon$-far from $\mathcal{P}$.
If the tester always answers correctly when $M$ satisfies $\mathcal{P}$, we say that the tester has
a one-sided error.
We say that $\mathcal{P}$ is testable if there is a one-sided error tester for $\mathcal{P}$ that makes a constant number of queries (that depends only on $\mathcal{P}$ and $\epsilon$ but not on the size of the input).
Furthermore, $\mathcal{P}$ is \emph{easily testable} if the number of queries is polynomial in $\epsilon^{-1}$.
Clearly, any hereditary property of matrices is testable by Theorem \ref{thm:ABF_matrix_removal_lemma}, while any property $\mathcal{P}$ for which $f_{\mathcal{P}}(\epsilon)$ is shown to be polynomial in $\epsilon^{-1}$ is easily testable.

\subsection{Background and main results}
\label{subsec:main_results}
The results here are stated and proved for square $n \times n$ matrices, but can be generalized to non-square matrices in a straightforward manner.
Our first main result is an efficient \emph{weak removal lemma} for binary matrices.

\begin{theorem}
	\label{thm:disjoint_copies}
	If an $n \times n$ binary matrix $M$ contains $\epsilon n^2$ pairwise-disjoint copies of an $s \times t$ binary matrix $A$, then the total number of $A$-copies in $M$ is at least $\delta n^{s+t}$, where $\delta^{-1}$ is polynomial in $\epsilon^{-1}$.
\end{theorem}
Here a set of pairwise-disjoint $A$-copies in $M$ is a set of $s \times t$ submatrices of $M$, all equal to $A$, such that any entry of $M$ is contained in at most one of the submatrices. 

Theorem \ref{thm:disjoint_copies} is an analogue for binary matrices of the non-induced graph removal lemma. However, in the graph removal lemma, $\delta^{-1}$ is not polynomial in $\epsilon^{-1}$ in general, in contrast to the situation in Theorem \ref{thm:disjoint_copies}. 

Alon, Fischer and Newman \cite{AlonFischerNewman2007} proved an efficient induced
removal lemma for a certain type of finite families $\mathcal{F}$ of binary matrices. A
family $\mathcal{F}$ of matrices, or equivalently, a hereditary matrix property $\mathcal{P}_{\mathcal{F}}$, is \emph{closed under row (column) permutations}
if for any $A \in \mathcal{F}$, any matrix created by permuting the rows (columns
respectively) of $A$ is in $\mathcal{F}$. $\mathcal{F}$ is \emph{closed under permutations} if it is
closed under row permutations and under column permutations.
\begin{theorem}[\cite{AlonFischerNewman2007}]
	\label{thm:AlonFischerNewman1}
	Let $\mathcal{F}$ be a finite family of binary matrices that is closed under permutations.
	For any $\epsilon > 0$ there exists $\delta > 0$, where $\delta^{-1}$ is polynomial in $\epsilon^{-1}$, such that any $n \times n$ binary
	matrix that is $\epsilon$-far from $\mathcal{F}$-freeness contains $\delta n^{s+t}$ 
	copies of some $s \times t$ matrix $A \in \mathcal{F}$. 
\end{theorem}
The main consequence of Theorem \ref{thm:AlonFischerNewman1} is an efficient induced
removal lemma for bipartite graphs. Indeed, when representing a bipartite graph
by its (bi-)adjacency matrix, a forbidden subgraph $H$ is represented by the
family $\mathcal{F}$ of all matrices that correspond to bipartite graphs isomorphic to
$H$. Note that $\mathcal{F}$ is indeed closed under permutations in this case.
Thus, any hereditary bipartite graph property 
characterized by a finite set of forbidden induced subgraphs is easily testable. 

The problem of understanding whether the statement of Theorem \ref{thm:AlonFischerNewman1}
holds for \emph{any} finite family $\mathcal{F}$ of binary matrices,
was raised in \cite{AlonFischerNewman2007} and is still open. Only recently in \cite{AlonBenEliezerFischer2017} it was shown that the statement holds if
we ignore the polynomial dependence, as stated in Theorem \ref{thm:ABF_matrix_removal_lemma}.
\begin{open}
	\label{open:open1}
	Is it true that for any fixed finite family $\mathcal{F}$ of binary matrices and any $\epsilon > 0$, there exists $\delta > 0$ with $\delta^{-1}$ polynomial in $\epsilon^{-1}$, 
	such that any $n \times n$ binary matrix $M$ that is $\epsilon$-far from $\mathcal{F}$-freeness
	contains $\delta n^{s+t}$ copies of some $s \times t$ matrix $A \in \mathcal{F}$?
\end{open}
Theorem \ref{thm:disjoint_copies} implies that to settle Problem \ref{open:open1} it is
enough to show the following. Fix a finite family $\mathcal{F}$ of binary matrices. Then
for any $\epsilon > 0$ there exists $\tau > 0$, with $\tau^{-1}$ polynomial in $\epsilon^{-1}$, such that any $n \times n$ binary matrix that is $\epsilon$-far from $\mathcal{F}$-freeness contains $\tau n^2$ pairwise disjoint copies of matrices from $\mathcal{F}$. 

Our second main result makes progress towards solving Problem \ref{open:open1} by
generalizing the statement of Theorem \ref{thm:AlonFischerNewman1} to any family $\mathcal{F}$
of binary matrices that is closed under row (or column) permutations.
From now on we only state the results for families that are closed under row permutations, but analogous results hold for families closed under column permutations.
\begin{theorem}
	\label{thm:removal_row_perm}
	Let $\mathcal{F}$ be a finite family of binary matrices that is closed under row permutations.
	For any $\epsilon > 0$ there exists $\delta > 0$, where $\delta^{-1}$ is polynomial in $\epsilon^{-1}$, such that any $n \times n$ binary
	matrix that is $\epsilon$-far from $\mathcal{F}$-freeness contains $\delta n^{s+t}$ 
	copies of some $s \times t$ matrix $A \in \mathcal{F}$.
\end{theorem}
\begin{cor}
	\label{cor:reasonably_row_ordered}
	Any hereditary property of binary matrices that is characterized by a finite forbidden 
	family closed under row permutations is easily testable.
\end{cor}

Our proof of Theorem \ref{thm:removal_row_perm} is somewhat simpler than the
original proof of Theorem \ref{thm:AlonFischerNewman1}. One of the main tools in the
proofs of Theorems \ref{thm:disjoint_copies} and \ref{thm:removal_row_perm} is an
efficient conditional regularity lemma for matrices developed in
\cite{AlonFischerNewman2007} (see also \cite{LovaszSzegedy2010}). In the proof
of Theorem \ref{thm:removal_row_perm} we only use a simpler form of the lemma, which
is also easier to prove. The statement of the lemma and the proofs of Theorems
\ref{thm:disjoint_copies}, \ref{thm:removal_row_perm} appear in Section
\ref{sec:binary}.


Besides the above two main results,  
we also describe a simpler variant of the construction of Fischer and Rozenberg \cite{FischerRozenberg2007}, showing that for ternary matrices, the dependence between the parameters is not polynomial in general. 
We further suggest a way to tackle the weak removal lemma (i.e.\@ the analogue of Theorem \ref{thm:disjoint_copies}, without the polynomial dependence)
in high dimensional matrices over arbitrary alphabets, by reducing it to an equivalent problem 
that looks more accessible. For more details, see Section \ref{sec:non_binary}.

\section{Related work}
\label{sec:previous}
Removal lemmas have been studied extensively in the context of graphs.
The non-induced graph removal lemma (which was stated in the beginning of Section \ref{sec:intro}) has been one of the first applications of the celebrated
Szemer\'edi graph regularity lemma \cite{Szemeredi1976}.
The \emph{induced graph removal lemma}, established in
\cite{AlonFischerKrivelevichSzegedy2000} by proving a stronger version of the
graph regularity lemma, is a similar result considering induced subgraphs. It states that for any finite family $\mathcal{F}$ of graphs and any $\epsilon > 0$ there exists $\delta = \delta(\mathcal{F}, \epsilon) > 0$ with the following property.
If an $n$-vertex graph $G$ is $\epsilon$-far from $\mathcal{F}$-freeness, then it contains at least
$\delta n^{v(F)}$ induced copies of some $F \in \mathcal{F}$.
Here $v(F)$ denotes the number of vertices in $F$, and 
$G$ is said to be (induced) $\mathcal{F}$-free if no induced subgraph of $G$ is isomorphic to a graph from 
$\mathcal{F}$.

The induced graph removal lemma was later extended to infinite families \cite{AlonShapira2008}, stating the following.
For any finite or infinite family $\mathcal{F}$ of graphs and any $\epsilon > 0$ there exists $f_{\mathcal{F}}(\epsilon)$ with the following property. If an $n$-vertex graph $G$ is
$\epsilon$-far from $\mathcal{F}$-freeness, then with probability at least $2/3$, a random induced subgraph of $G$ on $f_{\mathcal{F}}(\epsilon)$ vertices contains a graph from $\mathcal{F}$. 
Note that when $\mathcal{F}$ is finite, the statement of the infinite induced removal lemma is indeed equivalent to that of the finite version of the induced removal lemma.

The graph removal lemma was also extended to hypergraphs \cite{Gowers2001, RodlSkokan2004, NagleRodlSchacht2006, Tao2006}. 
See \cite{ConlonFox2013} for many more useful variants, quantitative strengthenings and extensions of the graph removal lemma.

Very recently, the authors and Fischer \cite{AlonBenEliezerFischer2017} generalized the
(finite and infinite) induced graph removal lemma by obtaining an \emph{order-preserving} version of it, and also 
showed that the same type of proof can be used to obtain a removal lemma for two-dimensional \emph{matrices} (with row and column order) over a finite alphabet; this it Theorem \ref{thm:ABF_matrix_removal_lemma} above. 

However, even for the non-induced graph removal lemma where the forbidden subgraph is a triangle, the best known general upper bound for $\delta^{-1}$ in terms of $\epsilon^{-1}$ is of tower-type \cite{Fox2011, ConlonFox2012}. On the other hand, the best known lower bound for the dependence is super-polynomial but sub-exponential, and builds on a construction of Behrend \cite{Behrend1946}. See \cite{Alon2001} for more details.
Understanding the ``right'' dependence of $\delta^{-1}$ in $\epsilon^{-1}$, even for the simple case where the forbidden graph $H$ is a triangle, is considered an important and difficult open problem.

In view of the above discussion, a lot of effort has been dedicated to the problem of characterizing the hereditary graph properties $\mathcal{P}$
for which $f_{\mathcal{P}}(\epsilon)$ is polynomial in $\epsilon^{-1}$, i.e.\@, the easily testable graph properties. See the recent work of Gishboliner and Shapira \cite{GishbolinerShapira2016}; for other previous works on this subject, see, e.g., \cite{Alon2001, AlonShapira2006, AlonFox2015}.
Our work also falls under this category, but for (ordered) matrices instead of graphs; it is the first work of this type for \emph{ordered} two-dimensional graph-like structures.

We finish by mentioning several other relevant removal lemma type results.
Removal lemmas for vectors (i.e.\@ one dimensional matrices where the order is
important) are generally easier to obtain; in particular, a removal lemma for
vectors over a fixed finite alphabet can be derived from a removal lemma for
regular languages proved in \cite{AlonKrivelevichNewmanSzegedy99}. A removal
lemma for partially ordered sets with a grid-like structure, which can be seen
as a generalization of the removal lemma for vectors, can be deduced from a
result of Fischer and Newman in \cite{FischerNewman2007}, where they mention
that this problem for submatrices is more complicated and not understood. 
Recently, Ben-Eliezer, Korman and Reichman \cite{BenEliezerKormanReichman2017} obtained 
a removal lemma for patterns in multi-dimensional matrices. A pattern must be
taken from \emph{consecutive} locations, whereas in our case the
rows and columns of a submatrix need not be consecutive.
The case of patterns behaves very differently than that of submatrices,
and in particular, in the removal lemma for patterns the parameters are linearly
related (for any alphabet size) unlike the case of submatrices (in which,
for alphabets of $3$ letters or more, the relation 
cannot be polynomial).

\section{Notation}
\label{sec:more_notation}
Here we give some more notation that will be useful throughout the rest of the
paper. We give the notation for rows but the notation for columns is equivalent.
Let $M:[m] \times [n] \to \Gamma$ be an $m \times n$ matrix.
For two rows in $M$ whose indices in $I$ are $r < r'$, we say that row $r$ is
\emph{smaller} than row $r'$ and row $r'$ is \emph{larger} than row $r$.
The \emph{predecessor} of row $r$ in $M$ is the largest row $\bar{r}$ in $M$
smaller than $r$. In this case we say that $r$ is the \emph{successor} of
$\bar{r}$.

Let $S$ be the submatrix of $M$ on $\{r_1, \ldots, r_s\} \times \{c_1, \ldots,
c_t\}$ where $r_1 < \ldots < r_s$ and $c_1 < \ldots < c_t$. For $i=1,\ldots, s$,
the \emph{$i$-row-index} of $S$ in $M$ is $r_i$; For two submatrices $S,S'$ of
the same dimensions and with $i$-row-indices $r_i, r'_i$ respectively we say
that $S$ is \emph{$i$-row-smaller} than $S'$ if $r_i < r'_i$ and
\emph{$i$-row-bigger} if $r_i > r'_i$.

Let $X = \{x_1, \ldots, x_{s-1}\} \subseteq [m]$ with $0 < x_1 < \ldots <
x_{s-1} < m$ and $Y = \{y_1, \ldots, y_{t-1}\} \subseteq [n]$ with $0 < y_1 <
\ldots < y_{t-1} < n$ be subsets of indices.
The submatrix $S$ is \emph{row-separated} by $X$ if $r_i \leq x_i < r_{i+1}$ for
any $i=1,\ldots, s-1$, \emph{column-separated} by $Y$ if $c_j \leq y_j <
c_{j+1}$ for any $j=1,\ldots, t-1$ and \emph{separated} by $X \times Y$ if it is
row separated by $X$ and column separated by $Y$.
The elements of $X,Y$ are called \emph{row separators}, \emph{column separators}
respectively.

\subsection{Folding and unfoldable matrices}
A matrix is \emph{unfoldable} if no two
neighboring rows in it are equal and no two neighboring columns in it are
equal. The \emph{folding} of a matrix $A$ is the unique matrix $\tilde{A}$
generated from $A$ by deleting any row of $A$ that is equal to its predecessor,
and then deleting any column of the resulting matrix that is equal to its
predecessor. Note that $\tilde{A}$ is unfoldable.
\begin{lemma}
	\label{lem:foldable_to_unfoldable_lem}
	Fix an $s \times t$ matrix $A$ and let $\tilde{A}$ be its $s' \times t'$ folding. For any $\epsilon > 0$ there exist $n_0, \delta > 0$, where $n_0$ and $\delta^{-1}$ are polynomial in $\epsilon^{-1}$, such that
	for any $n \geq n_0$, any $n \times n$ matrix $M$   that contains $\epsilon n^{s'+t'}$ copies of $\tilde{A}$ also contains $\delta n^{s+t}$ copies of $A$. 
\end{lemma}
Lemma \ref{lem:foldable_to_unfoldable_lem} implies that generally, to prove removal
lemma type results for finite families, it is enough to only consider families of unfoldable
matrices. The proof follows immediately from the following lemma.
\begin{lemma}
	\label{folding_one_step}
	Let $A$ be an $s \times t$ fixed matrix and
	let $A'$ be an $s' \times t$ matrix created from $A$ by deleting rows that 
	are equal to their predecessors in $A$.
	Then for any $\epsilon > 0$ there exist $n_1 = n_1(A, \epsilon) > 0$ and $\tau =
	\tau(A, \epsilon) > 0$, where $n_1$ and $\tau^{-1}$ are polynomial in $\epsilon^{-1}$, 
	such that for any $n \geq n_1$, any $n \times n$ matrix
	$M$ that contains $\epsilon n^{s'+t}$ copies of $A'$ also contains $\tau n^{s+t}$
	copies of $A$.
\end{lemma}

\begin{proof}[Proof of Lemma \ref{folding_one_step}]
	Let $T$ be the family of all $n \times t$ submatrices $S$ of $M$ containing at
	least $\epsilon {n^{s'}} / 2$ copies of $A'$.
	Any $S \in T$ has $\binom{n}{s'} \leq n^{s'}$ $s' \times t$ submatrices, so the number of
	$A'$ copies in submatrices from $T$ is at most $|T| n^{s'}$. On the other
	hand, there are $\binom{n}{t} \leq n^t$ $n \times t$ submatrices of $M$ so the number of
	$A'$ copies in $n \times t$ submatrices not in $T$ is less than $\epsilon n^{s'+t} / 2$. Hence the total number of $A'$ copies in
	submatrices from $T$ is at least $\epsilon n^{s'+t} / 2$,
	implying that $|T| \geq \epsilon n^t / 2$.
	
	Observe that any $S \in T$ contains a collection $\mathcal{A}(S)$ of $\epsilon n / 2s'$
	pairwise disjoint copies of $A'$. To show this, we follow a greedy approach,
	starting with a collection $\mathcal{B}$ of all $A'$-copies in $S$ and with
	empty $\mathcal{A}$. As long as $\mathcal{B}$ is not empty, we arbitrarily
	choose a copy $C \in \mathcal{B}$ of $A'$, add $C$ to $\mathcal{A}$ and delete
	all $A'$-copies intersecting $C$ (including itself) from $\mathcal{B}$. In each
	step, the number of deleted copies is at most $s' n^{s'-1}$, so the number of steps is at least $\epsilon n^{s'} /
	2 s' {n^{s'-1}} = \epsilon n / 2s'$.
	
	Let $\delta = \epsilon / 5ss'$ and take $S \in T$. Assuming that $n$ is large enough,
	pick disjoint collections $\mathcal{A}_1, \ldots, \mathcal{A}_s \subseteq \mathcal{A}(S)$, each of size at least $\delta n$,
	so that all $A'$-copies in $\mathcal{A}_i$ are $i$-row-smaller than all $A'$-copies in $\mathcal{A}_{i+1}$ for any $1 \leq i \leq s-1$.
	Then there are $\delta^s n^s$ copies of $A$ in $S$: Each $s \times t$ submatrix of $S$ whose $i$-th
	row is taken as the $i$-th row of a matrix from $\mathcal{A}_i$ is equal to $A$.
	Therefore, the total number of $A$-copies in $M$ is at least $|T| \delta^s n^s \geq \epsilon \delta^s n^{s+t} / 2$, as desired.
\end{proof}



\section{Proofs for the binary case}
\label{sec:binary}
This section is dedicated to the proof of our main results in the binary domain:
Theorem \ref{thm:disjoint_copies} and Theorem \ref{thm:removal_row_perm}.
As a general remark for the proofs in this section, 
We may and will assume that a square matrix $M$ is sufficiently large (given
$\epsilon > 0$), by which we mean that $M$ is an $n \times n$ matrix with $n \geq
n_0$ for a suitable $n_0 > 0$ that is polynomial in $\epsilon^{-1}$.

One of the main tools in the proofs of this section is a conditional regularity
lemma for matrices due to Alon, Fischer and Newman \cite{AlonFischerNewman2007}.
We describe a simpler version of the lemma (this is Lemma
\ref{conditional_clustering} below) along with another useful result from their
paper (Lemma \ref{conditional_partition} below). Combining these results
together yields the original version of the conditional regularity lemma used in
the original proof of Theorem \ref{thm:AlonFischerNewman1} in
\cite{AlonFischerNewman2007}. It is worth to note that even though Theorem
\ref{thm:removal_row_perm} generalizes Theorem \ref{thm:AlonFischerNewman1}, for its
proof we only need the simpler Lemma \ref{conditional_clustering} and not the
original regularity lemma, whose proof requires significantly more work. Lemma
\ref{conditional_partition} is only used in the proof of Theorem
\ref{thm:disjoint_copies}.

We start with some definitions. 
A $(\delta, r)$-\emph{row-clustering} of an $n \times n$ matrix $M$ is a
partition of the set of rows of $M$ into $r+1$ \emph{clusters} $R_0, \ldots,
R_r$ such that the \emph{error cluster} $R_0$ satisfies $|R_0| \leq \delta n$
and for any $i=1, \ldots, r$, every two rows in $R_i$ differ in at most $\delta
n$ entries. That is, for every $e,e' \in R_i$, one can make row $e$ equal to
$e'$ by modifying at most $\delta n$ entries. A \emph{$(\delta,
	r)$-column-clustering} is defined analogously on the set of columns of $M$.
The first conditional regularity lemma states the following.
\begin{lemma}[\cite{AlonFischerNewman2007}]
	\label{conditional_clustering}
	Let k be a fixed positive integer and let $\delta > 0$ be a small real. For
	every $n \times n$ binary matrix $M$ with $n > (k/\delta)^{O(k)}$, either $M$
	admits $(\delta, r)$-clusterings for both the rows and the columns with $r \leq
	(k/\delta)^{O(k)}$, or for every $k \times k$ binary matrix $A$, at least a
	$(\delta/k)^{O(k^2)}$ fraction of the $k \times k$ submatrices of $M$ are copies
	of $A$.
\end{lemma}
Let $R$ be a set of rows and let $C$ be a set of columns in an $n \times n$
matrix $M$. The \emph{block} $R \times C$ is the submatrix of $M$ on $R \times
C$.
A block $B$ is \emph{$\delta$-homogeneous} with \emph{value} $b$ if there exists
$b \in \{0,1\}$ such that at least a $1-\delta$ fraction of the entries of $B$
are equal to $b$.
A $(\delta, r)$-\emph{partition} of $M$ is a couple $(\mathcal{R}, \mathcal{C})$
where $\mathcal{R} = \{ R_1, \ldots, R_r\}$ is a partition of the set of rows
and $\mathcal{C} = \{ C_1, \ldots, C_r\}$ is a partition of the set of columns
of $M$, such that all but a $\delta$-fraction of the entries of $M$ lie in
blocks $R_i \times C_j$ that are $\delta$-homogeneous.
The second result that we need from \cite{AlonFischerNewman2007}, relating
clusterings and partitions of a matrix, is as follows.
\begin{lemma}[\cite{AlonFischerNewman2007}]
	\label{conditional_partition}
	Let $\delta > 0$. If a square binary matrix $M$ has $(\delta^2/16,
	r)$-clusterings $\mathcal{R}, \mathcal{C}$ of the rows and the columns
	respectively then $(\mathcal{R}, \mathcal{C})$ is a $(\delta, r+1)$-partition of
	$M$.
\end{lemma}
For the proofs of the above lemmas see \cite{AlonFischerNewman2007}.
We continue to the proof of Theorem \ref{thm:disjoint_copies}. The following lemma
is a crucial part of the proof.
\begin{lemma}
	\label{lem:separated}
	Fix an $s \times t$ matrix $A$. For any $\epsilon > 0$ there exists $\tau > 0$,
	where $\tau^{-1}$ is polynomial in $\epsilon^{-1}$,
	such that any $n \times n$ matrix $M$ containing $\epsilon n^2$ pairwise-disjoint copies of $A$ either
	contains $\tau n^{s+t}$ copies of any $s \times t$ matrix, or there exist
	subsets of indices $X,Y$ of sizes $s-1, t-1$ respectively such that $M$ contains
	$\tau n^2$ pairwise disjoint copies of $A$ that are separated by $X \times Y$.
\end{lemma}
Before providing the full proof of Lemma \ref{lem:separated}, we present a sketch of the proof.
	Clearly, whenever we apply Lemma \ref{conditional_clustering} throughout the proof, we may assume that the outcome is that $M$ has suitable row and column clusterings, as the other possible outcome of Lemma \ref{conditional_clustering} finishes the proof immediately.
	The main idea of the proof is to gradually find row separators,
	and then column
	separators, while maintaining a large set of pairwise disjoint copies of $A$ that
	conform to these separators. This is done inductively (first for the rows, and then for the columns). The inductive step is described in what follows.
	
	Assume we currently have $j-1 \geq 0$ row-separators, and a set $
	\mathcal{A}$ of many pairwise disjoint $A$-copies that have their first $j$
	rows separated by these row-separators. We take a clustering of the rows of $M$, and consider a cluster in which many rows are ``good'', in the sense that they contain the $j$-th row of many of the disjoint $A$-copies from $\mathcal{A}$. 
	We put our $j$-th separator as the medial row among the good rows.
	Next, we consider a matching of pairs $(r_1, r_2)$ of good rows, where in each such pair $r_1$ lies before the $j$-th separator and $r_2$ lies after the $j$-th separator. Observe that all good rows lie after the $(j-1)$-th separator.
	
	If we take all pairwise-disjoint $A$-copies from $\mathcal{A}$ whose $j$-th row is $r_2$, and ``shift'' their $j$-th row to be $r_1$, then most of them will still be $A$-copies (as rows $r_1$ and $r_2$ are very similar, since they are in the same row cluster).
	This process creates a set $\mathcal{A}'$ of many pairwise disjoint $A$-copies whose $i$-th row lies between separators $i-1$ and $i$ for any $i \leq j$, and the $(j+1)$-th row lies after separator $j$. This finishes the inductive step. 

We now continue to the full proof of Lemma \ref{lem:separated}.

\begin{proof}[Proof of Lemma 
\ref{lem:separated}]
	Let $\epsilon > 0$ and let $M$ be a large enough $n
	\times n$ binary matrix containing a collection $U_0$ of $\epsilon n^2$ pairwise disjoint $A$-copies.
	
	We prove the following claim by induction on $i$, for $i=0,1,\ldots,s-1$:
	there exist $\tau_i, \delta_i$ with $\tau_i^{-1},\delta_i^{-1}$ polynomial in
	$\epsilon^{-1}$ such that either
	 $M$ contains $\tau_i n^{s+t}$ copies of any $s \times t$
	matrix or there exist $0 = x_0 < x_1 < \ldots < x_i$ and a set $U_i$ of
	$\delta_i n^2$ pairwise disjoint $A$-copies in $M$ whose $j$-th row is bigger
	then $x_{j-1}$ and no bigger than $x_j$ for any $1 \leq j \leq i$, and the
	$(i+1)$-th row is bigger than $x_i$. The base case $i=0$ is trivial with
	$\delta_0 = \epsilon$. Suppose now that $i \geq 1$ and that $x_0, \ldots,
	x_{i-1}$, $\delta_{i-1}$ and $U_{i-1}$ are already determined. Applying lemma
	\ref{conditional_clustering} on $M$ with parameters $k = \max\{s,t\}$ and
	$\delta_{i-1}/4$, either $M$ contains $\tau_i n^{s+t}$ copies of any $s \times t$ matrix
	with $\tau_i^{-1}$ polynomial in $\epsilon^{-1}$ and we are done, or $M$ has a
	$(\delta_{i-1}/4, r_i)$-row-clustering $\mathcal{R}_i$ of $M$ for $r_i$
	polynomial in $\delta_{i-1}^{-1}$ and so in $\epsilon^{-1}$. The number of rows
	of $M$ that contain the $i$-th row of at least $\delta_{i-1}n /2$ of the
	$A$-copies in $U_{i-1}$ is at least $\delta_{i-1} n / 2$, since the number of
	$A$-copies in $U_{i-1}$ whose $i$-th row is not taken from such a row of $M$ is
	less that $n \cdot \delta_{i-1}n /2 = \delta_{i-1}n^2/2$. Let $R_i$ be a row
	cluster that contains at least $\delta_{i-1} n/ 2r_i$ such rows. Note that all
	of these rows are bigger than $x_{i-1}.$ Take subclusters $R_i^1, R_i^2$ of
	$R_i$, each containing at least $\lfloor \delta_{i-1} n/ 4r_i \rfloor \geq
	\delta_{i-1} n/ 5r_i $ such rows (the inequality holds for $n$ large enough)
	where each row in $R_i^1$ is smaller than each row in $R_i^2$. Take $x_i$ to be
	the row index of the biggest row in $R_i^1$.
	
	Take arbitrarily $\delta_{i-1} n/ 5r_i$ couples of rows $(r,r')$ where $r \in
	R_i^2$ and $r' \in R_i^1$ and every row participates in at most one couple. Let
	$(r,r')$ be such a couple. There exist $\delta_{i-1} n / 2$ $s \times t$
	submatrices of $M$ that are $A$-copies from $U_{i-1}$ and whose $i$-th row is
	$r$. Moreover, for any $j < i$ the $j$-th row of each of these submatrices lies
	between $x_{j-1}$ (non-inclusive) and $x_j$ (inclusive). Since $r$ and $r'$
	differ in at most $\delta_{i-1} n / 4$ entries, there are at least $\delta_{i-1}
	n / 4$ such submatrices $T$ that satisfy the following: If we modify $T$ by
	taking its $i$-th row to be $r'$ instead of $r$, $T$ remains an $A$-copy.
	Moreover, after the modification, the $i$-th row of $T$ is in $R_i^1$ and is
	therefore no bigger than $x_i$, whereas the $(i+1)$-th row of $T$ is bigger than
	the $i$-th row of $T$ before the modification which is bigger than $x_i$, as
	needed.
	For every couple $(r,r')$ we can produce $\delta_{i-1} n / 4$ pairwise disjoint
	copies of $A$ whose $j$-th row is between $x_{j-1}$ and $x_j$ for any $j \geq i$
	and the $(i+1)$-th row is after $x_i$. There are $\delta_{i-1} n/ 5r_i$ such
	couples $(r,r')$, and in total we get a set $U_i$ of
	$\delta_i n^2$ copies of $A$ with the desired structure for $\delta_i =
	\delta_{i-1}^2 / 20r_i$ where $\delta_i^{-1}$ is polynomial in
	$\delta_{i-1}^{-1}$ and so in $\epsilon^{-1}$.
	Note that the copies in $U_i$ are pairwise disjoint.
	In the end of the process there is a set $U = U_s$ of $\delta_s n^2$ pairwise
	disjoint copies of $A$ whose rows are separated by $X = \{x_1, \ldots,
	x_{s-1}\}$. A feature that is useful in what follows is that each copy in $U$
	has exactly the same set of columns (as a submatrix of $M$) as one of the
	original copies of $U_0$.
	
	Now we apply the same process as above but in columns instead of rows, starting
	with the $\delta_s n^2$ copies in $U$. In the end of the process, we obtain that
	for some $\hat{\tau}_t, \hat{\delta}_t$ such that $\hat{\tau}_t^{-1}$ and
	$\hat{\delta}_t^{-1}$ are polynomial in $\delta_s^{-1}$ and so in
	$\epsilon^{-1}$, either $M$ contains $\hat{\tau}_t n^{s+t}$ copies of any $s \times t$ matrix, or there exists a set $\hat{U}$ of $\hat{\delta}_t n^2$ pairwise disjoint
	copies of $A$ whose columns are separated by a set of indices $Y$ of size $t-1$.
	Moreover, by the above feature, each of the copies in $\hat{U}$ has the same set
	of rows as some copy of $A$ from $U$, so each copy has its rows separated by
	$X$. Hence $X \times Y$ separates all copies in $\hat{U}$. Taking $\tau =
	\min\{\hat{\tau}_t, \hat{\delta}_t\}$ finishes the proof.
\end{proof}
Next we show how Theorem \ref{thm:disjoint_copies} follows from Lemma
\ref{lem:separated}.
The idea of the proof is to show, using Lemmas \ref{conditional_partition} and \ref{lem:separated}, that there is a partition of $M$ with blocks $R_i \times C_j$ (for $1 \leq i \leq s$, $1 \leq j \leq t$) satisfying the following.
\begin{itemize}
\item All row clusters $R_i$  and all column clusters $C_j$ are large enough.
\item All rows of $R_i$ ($C_j$) lie before all rows (columns) of $R_{i+1}$ ($C_{j+1}$ respectively) for any $i$ and $j$.
\item $R_i \times C_j$ is almost homogeneous, and its ``popular'' value is $A_{ij}$.
\end{itemize}
Using these properties, it is easy to conclude that $M$ contains many $A$-copies.

We now complete the proof of Theorem \ref{thm:disjoint_copies}.

\begin{proof}[Proof of Theorem \ref{thm:disjoint_copies}]
	Let $A$ be an $s \times t$ binary matrix and let $k = \max\{s,t\}$. Let $\epsilon > 0$ and let $M$ be a large enough $n
	\times n$ binary matrix that contains $\epsilon n^2$ pairwise disjoint $A$-copies.
	Lemma \ref{lem:separated} implies that either 
	$M$ contains $\tau n^{s+t}$ copies of $A$ where
	$\tau^{-1}$ is polynomial in $\epsilon^{-1}$ (in this case we are done), or $M$
	contains at least $\tau n^2$ pairwise disjoint copies of $A$ separated by $X
	\times Y$ for suitable index subsets $X,Y$. By Lemma
	\ref{conditional_clustering} we get that either $M$ has
	$(\tau^2/128,r)$-clusterings of the rows and the columns where $r$ is polynomial
	in $\tau^{-1}$ and so in $\epsilon^{-1}$, or at least a $\zeta =
	(\tau^2/128k)^{O(k^2)}$ fraction of the $s \times t$ submatrices are $A$; in the
	second case we are done. Suppose then that $M$ has $(\tau^2/128,r)$-clusterings
	$\mathcal{R}, \mathcal{C}$ of the rows, columns respectively.
	The next step is to create refinements of the clusterings. Write the elements of
	$X$ as $x_1 < \ldots < x_{s-1}$ and let $x_0 = 0, x_s = n$. Partition each $R
	\in \mathcal{R}$ into $s$ parts where the $i$-th part for $i=1,\ldots,s$
	consists of all rows in $R$ with index at least $x_{i-1}$ and less than $x_i$.
	Each such part is also a $\tau^2/128$-cluster.
	Now separate each $C \in \mathcal{C}$ into $t$ parts in a similar fashion.
	This creates $(\tau^2/128, (r+1)k)$-clusterings $\mathcal{R'}, \mathcal{C'}$ of
	the rows and the columns respectively (where some of the clusters might be
	empty). By Lemma \ref{conditional_partition}, $\mathcal{P} = (\mathcal{R'},
	\mathcal{C'})$ is a $(\tau/4, r')$-partition of $M$ where $r' = (r+1)k+1$, and
	each block of the partition has all of its entries between two neighboring row
	separators from $X$ and between two neighboring column separators from $Y$.
	
	There are at most $\tau n^2 / 4$ entries of $M$ that lie in
	non-$\tau/4$-homogeneous blocks of $\mathcal{P}$ and at most $\tau n^2 / 4$
	entries of $M$ that lie in $\tau/4$-homogeneous blocks of $\mathcal{P}$ but do
	not agree with the value of the block. Therefore, the number of entries as above
	is no more than $\tau n^2 / 2$, and so there exists a set of $\tau n^2 / 2$
	pairwise disjoint copies of $A$ in $M$ separated by $X \times Y$ in which all
	the entries come from $\tau/4$-homogeneous blocks and agree with the value of
	the block in which they lie.
	Hence there exist sets of rows $R_1, \ldots, R_s \in \mathcal{R'}$ and sets of
	columns $C_1, \ldots, C_t \in \mathcal{C'}$ and a collection $\mathcal{A}$ of
	$\tau n^2 / 2(r')^{2k}$ pairwise disjoint 
	$A$-copies separated by $X \times Y$ such that for any
	$1 \leq i \leq s$, $1 \leq j \leq t$, the block $R_i \times C_j$ is
	$\tau/4$-homogeneous, has value $A(i,j)$, lies between row separators $x_{i-1}$
	and $x_i$ and between column separators $y_{j-1}$ and $y_j$, and contains the
	$(i,j)$ entry of any $A$-copy in $\mathcal{A}$. This implies that $|R_i|, |C_j|
	\geq \tau n / 2(r')^{2k}$ for any $1 \leq i \leq s$ and $1 \leq j \leq t$, So
	there are $(\tau / 2(r')^{2k})^{s+t}n^{s+t}$ $s \times t$ submatrices of $M$
	whose $(i,j)$ entry lies in $R_i \times C_j$ for any $i,j$. Picking such a
	submatrix $S$ at random, the probability that $S(i,j) \neq A(i,j)$ for a
	specific couple $i,j$ is at most $\tau / 4$; thus $S$ is equal to $A$ with
	probability at least $1 - st\tau / 4 > 1/2$ for small enough $\tau$. Hence the number of $A$-copies in $M$ is at least$(\tau / 2(r')^{2k})^{s+t} n^{s+t} / 2$.
\end{proof}

Next we give the proof of Theorem \ref{thm:removal_row_perm}. For the proof, recall
the definition of an unfoldable matrix and a folding of a matrix from Section
\ref{sec:more_notation}. A family of matrices is \emph{unfoldable} if all matrices in it
are unfoldable. The \emph{folding} of a finite family $\mathcal{F}$ of matrices is the set
$\tilde{\mathcal{F}} = \{\tilde{A} : A \in \mathcal{F}\}$ of the foldings of the matrices in $\mathcal{F}$.
Observe that $\tilde{\mathcal{F}}$ is unfoldable for any family $\mathcal{F}$.
Note that if $\mathcal{F}$ is closed under (row) permutations then $\tilde{\mathcal{F}}$ is also
closed under (row) permutations.

We start with a short sketch of the proof, before turning to the full proof:
As before, we may assume that our matrix $M$ has a row clustering with suitable parameters.
We may also assume that the forbidden family is unfoldable.
Consider a submatrix $Q$ of $M$ that contains exactly one ``representative'' row from any large enough row cluster. The crucial idea is that if $Q$ does not contain many $A$-copies, then $M$ is close to $\mathcal{F}$-freeness. Indeed, one can modify all rows in $M$ to be equal to rows from $Q$ without making many entry modifications, and after this modification, it is possible to eliminate all $\mathcal{F}$-copies in $M$ (without creating new $\mathcal{F}$-copies) by only modifying those columns in $M$ that participate in some $\mathcal{F}$-copy in $Q$; if $Q$ does not contain many $\mathcal{F}$-copies then the number of such columns is small.
Since the above statement is true for any possible choice of $Q$, we conclude that if $M$ is $\epsilon$-far from $\mathcal{F}$-freeness then it must contain many $A$-copies.

\begin{proof}[Proof of Theorem \ref{thm:removal_row_perm}]
	It is enough to prove the statement of the theorem only for
	\emph{unfoldable} families that are closed under row permutations. Indeed,
	suppose that Theorem \ref{thm:removal_row_perm} is true for all unfoldable families
	that are closed under row permutations.
	Let $\mathcal{F}$ be a family of binary matrices that is closed under row permutations and
	let $\tilde{\mathcal{F}}$ be its folding.
	Then for any $\epsilon > 0$ there exists $\tilde{\delta} > 0$ such that any
	square binary matrix $M$ which is $\epsilon$-far from $\tilde{\mathcal{F}}$-freeness 
	contains $\tilde{\delta} n^{s'+t'}$ copies of some $s' \times t'$ matrix $B \in \tilde{\mathcal{F}}$, where $\tilde{\delta}^{-1}$ is polynomial in
	$\epsilon^{-1}$.
	Thus, provided that $M$ is large enough (i.e.\@ that it is an $n \times n$ matrix where $n \geq n_0$ for a suitable choice of $n_0$ polynomial in $\epsilon^{-1}$), 
	we can apply Lemma \ref{lem:foldable_to_unfoldable_lem}
	to get that $M$ also contains $\delta n^{s+t}$ copies of the matrix $A \in \mathcal{F}$ whose folding is $B$, for a small enough $\delta > 0$ where $\delta^{-1}$ is polynomial in $\epsilon^{-1}$.
	
	Therefore, suppose that $\mathcal{F}$ is an
	unfoldable finite family of binary matrices that is closed under row
	permutations.
	Let $k$ be the maximal row or column dimension of a matrix from $\mathcal{F}$. Let
	$\epsilon > 0$ and apply Lemma \ref{conditional_clustering} with parameters $k$
	and $\epsilon/6$. Let $M$ be a large enough $n \times n$ matrix with $n
	> (k/\epsilon)^{O(k)}$, then either $M$ contains $\delta_2 n^{2k}$ copies of any $k \times k$
	matrix, where $\delta_2^{-1}$ is polynomial in $\epsilon^{-1}$, or $M$ has an $(\epsilon/6, r)$-clustering of
	the rows with $r$ polynomial in $\epsilon^{-1}$.
	In the first case we are done, so suppose that $M$ has an $(\epsilon/6, r)$-clustering
	$\mathcal{R} = \{R_0, \ldots, R_r\}$ of the rows where $R_0$ is the error
	cluster.
	
	Suppose that $M$ is $\epsilon$-far from $\mathcal{F}$-freeness. 
	We say that a cluster $R \neq R_0$ in $\mathcal{R}$ is \emph{large} if it
	contains at least $\epsilon n / 6 r$ rows. Note that the total number of entries
	that do not lie in large clusters is at most $\epsilon n / 6 + \epsilon n / 6 =
	\epsilon n / 3$.
	Pick an arbitrary row $r(R)$ from every large cluster $R \in \mathcal{R}$ and
	denote by $Q$ the submatrix of $M$ created by these rows.
	Let $\mathcal{A}(Q)$ be a collection of pairwise disjoint copies of matrices
	from $\mathcal{F}$ in $Q$ that has the maximal possible number of copies.
	Suppose to the contrary that $|\mathcal{A}| \leq \epsilon n / 3 k$ and let $C$
	be the set of all columns of $M$ that intersect a copy from $\mathcal{A}$, then
	$C$ contains no more than $\epsilon n / 3$ columns.
	We can modify $M$ to make it $\mathcal{F}$-free as follows: 
	First modify every row that lies in a large cluster $R \in \mathcal{R}$ to be
	equal to $r(R)$. Then pick some row $r$ of $Q$ and modify all rows that are not
	contained in large clusters to be equal to $r$. Finally do the following: As
	long as $C$ is not empty, pick a column $c \in C$ that has a neighbor
	(predecessor or successor) not in $C$ and modify $c$ to be equal to its
	neighbor, and then remove $c$ from $C$.
	
	It is not hard to see that since $\mathcal{F}$ is unfoldable and closed under row
	permutations, after these modifications $M$ is $\mathcal{F}$-free. Indeed, after the first
	and the second steps, all rows of $M$ are equal to rows from $Q$; the order of
	the rows does not matter since $\mathcal{F}$ is closed under row permutations. 
	Now each time
	that we modify a column $c \in C$ in the third step, all copies of matrices from
	$\mathcal{F}$ that intersect it are destroyed and no new copies are created. By the
	maximality of $\mathcal{A}$, any copy of a matrix from $\mathcal{F}$ in the original $Q$
	intersected some column from $C$, so we are done.
	The number of entry modifications needed in the first, second, third step
	respectively is at most $\epsilon n^2 / 6$, $\epsilon n^2 / 3$, $\epsilon n^2 /
	3$ and thus by making only $5\epsilon n^2 / 6$ modifications of entries of $M$
	we can make it $\mathcal{F}$-free, contradicting the fact that $M$ is $\epsilon$-far from $\mathcal{F}$-freeness.
	
	Let $Q$ be any matrix of representatives of the large row clusters as above.
	Then $Q$ contains a collection $\mathcal{A}$ of $\epsilon n / 3 k$ pairwise
	disjoint copies of matrices from $\mathcal{F}$. In particular, there exist a certain $s
	\times n$ submatrix $T$ of $Q$ and an $s \times t$ matrix $A(Q) \in \mathcal{F}$ such that
	at least $\epsilon n / 3 k |\mathcal{F}| r^s$ of the copies in $\mathcal{A}$ are
	$A$-copies that lie in $T$. The following elementary removal lemma implies that
	$T$ contains many $A$-copies.
	\begin{obs}
		\label{obs:removal_strip}
		Fix an $s \times t$ matrix $A$. For any $\epsilon > 0$ there exists $\delta > 0$
		such that if an $s \times n$ matrix $T$ contains $\epsilon n$ pairwise-disjoint $A$-copies, then the total number of $A$-copies in
		$T$ is at least $\delta n^t$, with $\delta^{-1}$ polynomial in $\epsilon^{-1}$.
	\end{obs}
	\begin{proof}
		Let $\epsilon > 0$ and let $T$ be a large enough $s \times n$ matrix containing $\epsilon n$ pairwise
		disjoint copies of $A$. We construct $t$ disjoint subcollections
		$\mathcal{A}_1, \ldots, \mathcal{A}_t$ of $\mathcal{A}$, each of size $\epsilon
		n / 2t \leq \lfloor \epsilon n / t \rfloor$, such that for any $i < j$, all
		copies in $\mathcal{A}_i$ are $i$-column-smaller than all copies in
		$\mathcal{A}_j$. This is done by the following process for $i=1, \ldots, t$:
		take $\mathcal{A}_i$ to be the set of the $\epsilon n / 2t$ $i$-smallest copies
		in $\mathcal{A}$ and delete these copies from $\mathcal{A}$.
		Now observe that any $s \times t$ submatrix of $T$ that takes its $i$-th column
		(for $i=1, \ldots, t$) as the $i$-th column of some copy from $\mathcal{A}_i$ is
		equal to $A$. There are $(\epsilon n / 2t)^{t}$ such submatrices among all $\binom{n}{t} \leq n^t$ $s \times t$ submatrices of $T$, and so $T$ contains $
		(\epsilon / 2t)^{t} n^t$ $A$-copies.
	\end{proof}
	
	Observation \ref{obs:removal_strip} implies that for $Q$ and $A(Q)$ as above, $Q$ contains $\gamma n^{s+t}$ $A$-copies where $\gamma^{-1}$ is polynomial in $(\epsilon / 3 k
	|\mathcal{F}| r^s)^{-1}$ and so in $\epsilon^{-1}$. Finally we show that $M$ contains $\delta n^{s+t}$ copies of some $A \in \mathcal{F}$ where $\delta^{-1}$ is polynomial in $\gamma^{-1}$ and so in
	$\epsilon^{-1}$, finishing the proof of the Theorem.
	For any large cluster $R \in \mathcal{R}$ let $R'$ be some subcluster that
	contains exactly $\lfloor \epsilon n / 6r \rfloor > 0$ rows. Let $\mathcal{R'} =
	\{R' : R \in \mathcal{R} \text{ is large}\}$ and
	note that an $\alpha$-fraction of the $k \times k$ submatrices $S$ of $M$ have
	all of their rows in subclusters from $\mathcal{R'}$ with no subcluster
	containing more than one row of $S$, where $\alpha^{-1}$ is polynomial in
	$\epsilon^{-1}$.
	Let $S$ be a random $k \times k$ submatrix of $M$. Conditioning on the event
	that $S$ satisfies the above property, we can assume that $S$ is chosen in the
	following way: First a random $Q$ is created by picking uniformly at random one
	representative from every $R' \in \mathcal{R'}$, and then $S$ is taken as a
	random $k \times k$ submatrix of $Q$. Let $A = A(Q)$ be defined as above. The probability that $S$
	contains a copy of $A$ is at least $\gamma$. That is, a random $k \times k$
	submatrix $S$ of $M$ contains a copy of a matrix from $\mathcal{F}$ with probability at
	least $\alpha \gamma$, so there exists an $s \times t$ matrix $A \in \mathcal{F}$ that is
	contained in a randomly chosen such $S$ with probability at least $\alpha \gamma
	/ |\mathcal{F}|$, so $M$ contains $\alpha \gamma \binom{n}{s} \binom{n}{t} /
	|\mathcal{F}| k^{2k}$ copies of some $A \in \mathcal{F}$: To see this, observe that we can choose a random $s \times t$
	submatrix $S'$ of $M$ by first picking a random $k \times k$ submatrix $S$ and
	then picking an $s \times t$ random submatrix $S'$ of $S$. The event that $S$
	contains a copy of $A$ has probability at least $\alpha \gamma / |\mathcal{F}|$, and
	conditioned on this event, $S'$ is equal to $A$ with probability at least
	$k^{-2k}$, as the number of $s \times t$ submatrices of $S$ is at most $s^k t^k
	\leq k^{2k}$. The proof is concluded by taking a suitable $\delta = \delta(\epsilon) > 0$ that satisfies $\delta n^{s+t} \leq \alpha \gamma \binom{n}{s} \binom{n}{t} /
	|\mathcal{F}| k^{2k}$ for large enough values of $n$. Note that indeed $\delta^{-1}$ is polynomial in $\epsilon^{-1}$.
\end{proof}

%

\section{Multi-dimensional matrices over arbitrary alphabets}
\label{sec:non_binary}
As opposed to the polynomial dependence in the above results on
binary matrices, Fischer and Rozenberg \cite{FischerRozenberg2007} 
showed that in analogous results for
ternary matrices, as well as binary three-dimensional matrices, 
the dependence is super-polynomial in general.
The proof builds on a construction of Behrend \cite{Behrend1946}.
For the ternary case, it gives the following.
\begin{theorem}[\cite{FischerRozenberg2007}]
	\label{thm:FischerRozenberg1}
	There exists a (finite) family $\mathcal{F}$ of $2 \times 2$ 
	binary matrices that is closed under permutations
	and satisfies the following.
	For any small enough $\epsilon > 0$, there exists an
	arbitrarily large $n \times n$ ternary matrix $M$ 
	that contains $\epsilon n^2$ pairwise-disjoint copies of matrices from $\mathcal{F}$,
	yet the total number of submatrices from $\mathcal{F}$ in $M$ is no more than
	$\epsilon^{-c\log{\epsilon}} n^4$ where $c > 0$ is an absolute
	constant.
\end{theorem}
Theorem \ref{thm:FischerRozenberg1} implies that an analogue of
Theorem \ref{thm:disjoint_copies} with polynomial dependence cannot be obtained 
when the alphabet is bigger than binary, even when $\mathcal{F}$ is a small finite
family that is closed under permutations. In Subsection
\ref{subsec:lower_bound} we describe another construction that establishes
Theorem \ref{thm:FischerRozenberg1}, which is slightly simpler
than the original construction in \cite{FischerRozenberg2007}.

In what follows, we focus on the problem of finding a ``weak'' removal lemma
analogous to Theorem \ref{thm:disjoint_copies} for matrices in more than two dimensions over an arbitrary alphabet. Here we do not try to optimize the dependence between the parameters, but rather to 
show that such a removal lemma exists. Note that in two dimensions this removal lemma follows from Theorem \ref{thm:ABF_matrix_removal_lemma}, but our results here suggest a direction to prove a weak high dimensional removal lemma without trying to generalize the heavy machinery used in \cite{AlonBenEliezerFischer2017} to the high dimensional setting.
Our main result here states that this problem is equivalent in some sense to the problem of 
showing that if a hypermatrix $M$ contains many pairwise-disjoint copies of a hypermatrix $A$, then it contains a ``wide'' copy of $A$; more details are given later.
In what follows, we use the term $d$-matrix to refer to a matrix in $d$ dimensions.
An \emph{$(n,d)$-matrix} is a $d$-matrix whose dimensions are $n \times \cdots \times n$.

A weak removal lemma for families of $d$-matrices that are closed under permutations follows easily
from the hypergraph removal lemma \cite{Gowers2001, RodlSkokan2004, NagleRodlSchacht2006, Tao2006} using a suitable construction. 
\begin{prop}
	\label{obs:obs_non_binary}
	Let $\Gamma$ be an arbitrary alphabet and let $\mathcal{F}$ be a finite family of $d$-matrices
	over $\Gamma$ that is closed under permutations (in all $d$ coordinates). 
	For any $\epsilon > 0$ there exists $\delta > 0$ such that the following holds.
	If an $(n,d)$-matrix $M$ over $\Gamma$ contains $\epsilon n^d$ pairwise disjoint copies of $d$-matrices from $\mathcal{F}$, then $M$ contains $\delta n^{s_1 + \cdots + s_d}$ copies of some $s_1 \times \ldots \times s_d$ matrix $A \in \mathcal{F}$.
\end{prop}
Note that Theorem \ref{thm:FischerRozenberg1} implies that the dependence of
$\delta^{-1}$ on $\epsilon^{-1}$ in Proposition \ref{obs:obs_non_binary} cannot be
polynomial.
The question whether the statement of Proposition \ref{obs:obs_non_binary} holds for
any finite family $F$ is open for $d$-matrices with $d > 2$. Here we state the question in the following
equivalent but simpler form.
\begin{open}
	\label{open:open_non_binary}
	Let $d > 2$ be an integer.
	Is it true that for any alphabet $\Gamma$, $s_1 \times \ldots \times s_d$ matrix $A$ over $\Gamma$ and	$\epsilon > 0$ there exists $\delta > 0$, 
	such that for any $(n,d)$-matrix $M$ over
	$\Gamma$ containing $\epsilon n^d$ pairwise-disjoint copies of $A$, the total number of $A$-copies in $M$ is at least $\delta n^{s_1 + \cdots + s_d}$?
\end{open}
Note that Theorem \ref{thm:disjoint_copies} settles the two-dimensional binary case of Problem \ref{open:open_non_binary} with
$\delta^{-1}$ polynomial in $\epsilon^{-1}$, and Theorem \ref{thm:ABF_matrix_removal_lemma} settles the two-dimensional case over any alphabet. On the other hand, $\delta^{-1}$ cannot be polynomial in $\epsilon^{-1}$ if $|\Gamma| > 2$ or $d > 2$.

Our main theorem in this domain shows that Problem \ref{open:open_non_binary} is
equivalent to another statement that looks more accessible.
We need the following definition to describe it.
Let $M:[n_1] \times \ldots \times [n_d] \to \Gamma$ and 
let $S$ be the submatrix of $M$ on the indices $\{r_1^1, \ldots, r_{s_1}^1\} \times \ldots \times \{r_1^d, \ldots, r_{s_d}^{d}\}$ where $r_j^i < r_{j+1}^{i}$ for any $1 \leq i \leq d$, $1 \leq j \leq s_i-1$.
The \emph{(i, j)-width} of $S$ (for $1 \leq i \leq d$ and $1 \leq j \leq s_i - 1$) is
$(r_{j+1}^{i} - r_{j}^{i}) / n_i$.

\begin{theorem}
	\label{thm:two_equivalent}
	The following statements are equivalent for any $d \geq 2$.
	\begin{enumerate}
		\item \label{s1} 
		For any alphabet $\Gamma$, $s_1 \times \ldots \times s_d$ matrix $A$ over $\Gamma$ and $\epsilon > 0$ there
		exists $\delta > 0$ such that for any $(n,d)$-matrix $M$ that contains $\epsilon n^d$ pairwise disjoint copies of $A$, the total number of
		$A$-copies in $M$ is at least $\delta n^{s_1 + \cdots + s_d}$.
		
		\item \label{s2} 
		For any alphabet $\Gamma$, $s_1 \times \ldots \times s_d$ matrix $A$ over $\Gamma$ and $\epsilon > 0$ there
		exists $\delta > 0$ such that for any $(n,d)$-matrix $M$
		that contains $\epsilon n^d$ pairwise disjoint copies of $A$,  and
		any $1 \leq i \leq d$, $1 \leq j \leq s_i$, 
		there exists an $A$-copy in $M$ whose $(i,j)$-width is at least $\delta$.
	\end{enumerate}
\end{theorem}

The proofs of the statements here are given, for simplicity, only for two dimensional matrices, but they translate directly to higher dimensions. The only major difference in the high dimensional case is the use of the hypergraph removal lemma instead of the graph removal lemma.

We start with the (simple) proof of Proposition \ref{obs:obs_non_binary}. The proof
uses the non induced graph removal lemma. Some definitions are required for
the proof. An $s \times t$ \emph{reordering} $\sigma$ is a permutation of $[s]
\times [t]$ that is a Cartesian product of two permutations $\sigma_1:[s] \to
[s]$ and $\sigma_2:[t] \to [t]$. Given an $s \times t$ matrix $A$, the $s \times
t$ matrix $\sigma(A)$ is the result of the following procedure: First reorder
the rows of $A$ according to the permutation $\sigma_1$ and then reorder the
columns of the resulting matrix according to the permutation $\sigma_2$.

\begin{proof}[Proof of Proposition \ref{obs:obs_non_binary}]
	Let $k(\mathcal{F})$ denote the largest row or column dimension of matrices from $\mathcal{F}$. Let
	$\epsilon > 0$ and let $M$ be an $n \times n$ matrix over $\Gamma$ that contains  $\epsilon n^2$ pairwise-disjoint copies of matrices from $\mathcal{F}$. In particular, there is an $s
	\times t$ matrix $A \in \mathcal{F}$ such that $M$ contains $\epsilon n^2 / |\mathcal{F}|$ pairwise-disjoint copies of $A$. 
	
	We construct an $(s+t)$-partite graph $G$ on $(s+t)n$ vertices as follows: There
	are $s$ row parts $R_1, \ldots, R_s$ and $t$ column parts $T_1, \ldots, T_t$,
	each containing $n$ vertices. The vertices of $R_i$ ($C_i$) are labeled $r_1^i,
	\ldots, r_n^i$ ($c_1^i, \ldots, c_n^i$ respectively).
	Two vertices $r_i^a$ and $r_j^b$ (or $c_i^a$ and $c_j^b$) with $a \neq b$ are
	connected by an edge iff $i \neq j$.
	$r_i^a$ and $c_j^b$ are connected iff $M(i,j) = A(a,b)$.
	
	We now show that there exists a bijection between copies of $K_{s+t}$ in $G$ and
	couples $(S, \sigma)$ where $S$ is an $s \times t$ submatrix of $M$ and $\sigma$
	is an $s \times t$ reordering such that $\sigma(S) = A$.
	Indeed, take the following mapping: A couple $(S,\sigma)$, where $S$ is the
	submatrix of $M$ on $\{a_1, \ldots, a_s\} \times \{b_1, \ldots, b_t\}$ with $a_1
	< \ldots < a_s$ and $b_1 < \ldots < b_t$ and $\sigma = \sigma_1 \times
	\sigma_2$, is mapped to the induced subgraph of $G$ on $\{r_{a_1}^{\sigma_1(1)},
	\ldots, r_{a_s}^{\sigma_1(s)}, c_{b_1}^{\sigma_2(1)}, \ldots,
	c_{b_t}^{\sigma_2(t)}\}$.
	
	It is not hard to see that $(S,\sigma)$ is mapped to a copy of $K_{s+t}$ if and
	only if $\sigma(S)$ is equal to $A$. On the other hand, every copy of $K_{s+t}$
	in $G$ has exactly one vertex in each row part and in each column part, and
	there exists a unique couple $(S,\sigma)$ mapped to it.
	
	There exist $\epsilon n^2/ |\mathcal{F}|$ pairwise-disjoint
	$A$-copies in $M$ that are mapped (with the identity reordering) to
	edge-disjoint copies of $K_{s+t}$ in $G$. By the graph removal lemma, there
	exists $\delta > 0$ such that at least a $\delta$-fraction of the subgraphs of
	$G$ on $s+t$ vertices are cliques. Therefore, at least a $\delta$-fraction of
	the possible couples $(S,\sigma)$ (where $S$ is an $s \times t$ submatrix of $M$
	and $\sigma$ is an $s \times t$ reordering) satisfy $\sigma(S) = A$, concluding the proof.
\end{proof}
Next we give the proof of Theorem \ref{thm:two_equivalent}. We may and will assume
throughout the proof that $M$ is an $n \times n$ matrix where $n$ is large
enough with respect to $\epsilon$.
The terms $i$-height and $j$-width correspond to $(1,i)$-width and $(2,j)$-width, respectively, in the definition given before the statement of Theorem \ref{thm:two_equivalent}.
\begin{proof}[Proof of Theorem \ref{thm:two_equivalent}]
	We start with deriving Statement \ref{s2} from Statement \ref{s1}; this
	direction is quite straightforward, while the other direction is more
	interesting.
	Fix an $s \times t$ matrix $A$. Let $\epsilon > 0$ and assume that Statement
	\ref{s1} holds. There exists $\delta = \delta(\epsilon)$, such that if $M$ contains 
	$\epsilon n^2$ pairwise-disjoint $A$-copies then 
	it contains $\delta n^{s+t}$ copies of $A$.
	To prove
	Statement \ref{s2} we can pick $\delta' = \delta'(\epsilon) > 0$ small enough
	such that for any large enough $n \times n$ matrix $M$, any $1 \leq i \leq s-1$
	and any $1 \leq j \leq t-1$, the fraction of $s \times t$ submatrices with
	$i$-height (or $j$-width) smaller than $\delta'$ among all $s \times t$
	submatrices is at most $\delta / 2$. 
	Fix an $1 \leq i \leq s-1$. This choice of $\delta'$ implies that any matrix $M$
	containing $\epsilon n^2$ pairwise disjoint $A$-copies also contains an $A$-copy with
	$i$-height at least $\delta'$. Similarly, for any $1 \leq j \leq t-1$ there is
	an $A$-copy with $j$-width at least $\delta'$.

	Next we assume that Statement \ref{s2} holds and prove Statement \ref{s1}. 
	Fix an $s \times t$ matrix $A$ over an alphabet $\Gamma$, let $\epsilon > 0$ and
	let $M$ be a large enough $n \times n$ matrix containing a collection $\mathcal{A}_0$ of $\epsilon n^2$ pairwise disjoint $A$-copies. We
	will show that there exist $\epsilon^* > 0$ that depends only on $\epsilon$,
	sets $X,Y$ of row and column separators respectively of sizes $s-1$ and $t-1$
	and a collection of $\epsilon^* n^2$ disjoint $A$-copies
	separated by $X \times Y$ in $M$. Then we will combine a simpler variant of the
	construction used in the proof of Proposition \ref{obs:obs_non_binary} with the
	graph removal lemma to show that $M$ contains $\delta n^{s+t}$ copies of $A$ for a suitable
	$\delta(\epsilon) > 0$.

	The number of $A$-copies in $M$ does not depend on the alphabet, so we may consider $A$
	and $M$ as matrices over the alphabet $\Gamma' = \Gamma \cup \{\alpha\}$ for
	some $\alpha \notin \Gamma$, even though all symbols in $A$ and $M$ are from
	$\Gamma$.
	Without loss of generality we assume that no two entries in $A$ are equal.
	
	Let $X_0 = \phi$, $\epsilon_0 = \epsilon$ and let $M_0$ be the following $n \times n$ matrix over $\Gamma'$: All $A$-copies in
	$\mathcal{A}_0$ appear in the same locations in $M_0$, and all other entries of
	$M_0$ are equal to $\alpha$. Clearly, any $A$-copy in $M_0$ also appears in $M$.
	Next, we construct iteratively for any $i=1, \ldots, s-1$ an $n \times n$ matrix
	$M_i$ over $\Gamma'$ that contains a collection $\mathcal{A}_i$ of $\epsilon_i
	n^2$ pairwise disjoint copies of $A$ where $\epsilon_i > 0$ depends only on
	$\epsilon_{i-1}$, such that all $A$-copies in $M_i$ also exist in $M_{i-1}$. 
	We also maintain a set $X_i$ of row
	separators whose elements are $x_1 < \ldots < x_i$, such that any entry of
	$M_i$ between $x_{j-1}$ and $x_{j}$ for $j=1, \ldots, i$ (where we define $x_0 =
	0, x_s = n$) is either equal to one of the entries of the $j$-th row of $A$ or
	to $\alpha$.
	
	The construction of $M_i$ given $M_{i-1}$ is done as follows. By Statement
	\ref{s2}, there exists $\delta_i = \delta_i(\epsilon_{i-1})$ such that any
	matrix $M'$ over $\Gamma'$ containing at least $\epsilon_{i-1} n^2 / 2$ copies 
	of $A$ also contains a
	copy of $A$ with $i$-height at least $\delta_i$.
	We start with a matrix $M'$ equal to $M_{i-1}$ and an empty $\mathcal{A}_i$, and
	as long as $M'$ contains a copy of $A$ with $i$-height at least $\delta_i$, we
	add it to $\mathcal{A}_i$ and modify (in $M'$) all entries of all $A$-copies
	from $\mathcal{A}_{i-1}$ that intersect it to $\alpha$.
	By the separation that $X_{i-1}$ induces on $M'$, each such copy has its $j$-th
	row between $x_{j-1}$ and $x_j$ for any $1 \leq j \leq i-1$.
	
	This process might stop only when at least $\epsilon_{i-1} n^2/2$ of the copies from $\mathcal{A}_{i-1}$ in $M'$ have one of their entries modified.
	Since in each step
	at most $st$ copies of $A$ are deleted from $M'$, in the end $\mathcal{A}_i$
	contains at least $\epsilon_{i-1} n^2 / 2 s t$ pairwise disjoint copies of $A$
	with $i$-height at least $\delta_i$.
	Pick uniformly at random a row index $x_i > x_{i-1}$. The probability that a
	certain copy of $A$ in $\mathcal{A}_i$ has its $i$-th row at or above $x_i$ and
	its $(i+1)$-th row below $x_i$ is at least $\delta_{i}$. Therefore, the expected
	number of $A$-copies in $\mathcal{A}_i$ with this property is at least
	$\epsilon_i n^2$ with $\epsilon_i = \delta_i \epsilon_{i-1} / 2 s t$, so there
	exists some $x_i$ such that at least $\epsilon_i n^2$ $A$-copies in
	$\mathcal{A}_i$ have their first $i+1$ rows separated by $X_i = X_{i-1} \cup
	\{x_i\}$; delete all other copies from $\mathcal{A}_i$.
	We construct $M_i$ as follows: All $A$-copies from $\mathcal{A}_i$ appear in
	the same locations in $M_i$, and all other entries of $M_i$ are equal to
	$\alpha$.
	
	After iteration $s-1$ we have a matrix $M_{s-1}$ with $\epsilon_{s-1}n^2$ copies
	of $A$ separated by $X = X_{s-1}$. We apply the same process in columns instead
	of rows, starting with the matrix $M_{s-1}$. The resulting matrix $M^*$ contains
	$\epsilon^* n^2$ pairwise disjoint copies of $A$ separated by $X \times Y$ where
	$Y$ consists of the column separators $y_1 < \ldots < y_{t-1}$, $\epsilon^*$
	depends on $\epsilon$, and $M^*$ only contains $A$-copies that appeared in the original $M$.
	
	Finally, construct an $(s+t)$-partite graph $G$ on $2n$ vertices as follows: The
	row parts are $R_1, \ldots, R_s$ and the column parts are $C_1, \ldots, C_t$
	where $R_i$ ($C_i$) contains vertices labeled $x_{i-1}+1, \ldots, x_i$
	($y_{i-1}+1, \ldots, y_i$ respectively) with $x_0 = y_0 = 0, x_s=y_t =n$. Any
	two row (column) vertices not in the same part are connected.
	Vertices $a \in R_i, b \in C_j$ are connected if and only if $M^*(a,b) =
	A(i,j)$.
	Clearly there exists a bijection between $A$-copies in $M^*$ and $K_{s+t}$
	copies in $G$ that maps disjoint $A$-copies to edge disjoint $K_{s+t}$-copies in
	$G$, so it contains $\epsilon^* n^2$ edge disjoint $(s+t)$-cliques. By the graph
	removal lemma there exists $\delta = \delta(\epsilon^*) > 0$ such that a
	$\delta$-fraction of the subgraphs of $G$ on $s+t$ vertices are cliques. Hence
	at least a $\delta$-fraction of the $s \times t$ submatrices of $M$ are equal to $A$.
\end{proof}

\subsection{Lower bound}
\label{subsec:lower_bound}
In this subsection we give an alternative constructive proof of Theorem
\ref{thm:FischerRozenberg1}.
Our main tool is the following result in additive number theory from
\cite{Alon2001}, based on a construction of Behrend \cite{Behrend1946}.
\begin{lemma}[\cite{Alon2001,Behrend1946}]
	For every positive integer $m$ there exists a subset $X \subseteq [m] =
	\{1,\ldots,m\}$ with no non-trivial solution to the equation $x_1 + x_2 + x_3 =
	3x_4$, where $X$ is of size at least
	\begin{equation}
	|X| \geq \frac{m}{e^{20 \sqrt{\log{m}}}}.
	\end{equation}
\end{lemma}
\begin{proof}[Proof of Theorem \ref{thm:FischerRozenberg1}]
	Consider the family $\mathcal{F} = \{A,B\}$ where \[A = \left( 
	\begin{array}{cc}
	1 & 0 \\
	0 & 1
	\end{array}
	\right), \ \ B = \left( 
	\begin{array}{cc}
	0 & 1 \\
	1 & 0
	\end{array}
	\right)
	,\]
	and observe that $\mathcal{F}$ is closed under permutations. 
	Let $m$ be a positive integer divisible by $10$ and let $X \subseteq
	[m/10]$ be a subset with no non-trivial solution to the equation $x_1 + x_2 +
	x_3 = 3x_4$ that is of maximal size. We construct the following $m \times m$
	ternary matrix $M$. For any $1 \leq i \leq m/5$ and any $x \in X$ we put a copy
	of $A$ in $M$ as follows:
	\begin{align*}
	M(i,i+x) = M(m/2 + i + 2x, m/2 + i + 3x) &= 1 \\
	M(i,m/2+i+3x) = M(m/2 + i + 2x,i+x) &= 0.
	\end{align*} 
	We set all other entries of $M$ to $2$. Let $\mathcal{A}$ be the collection of
	$q = m|X|/5 \geq m^2 / {50e^{20\sqrt{\log{m}}}}$ pairwise disjoint copies of $A$
	in $M$ that are created as above.
	Note that all $A$-copies in $M$ are separated by $\{n/2 \} \times \{n/2 \}$,
	where there are two opposite quarters (with respect to the separation) that do
	not contain the entry $0$ and the two other opposite quarters do not contain
	$1$. Hence, every $A$-copy must contain one entry from each quarter, and $M$ does not contain copies of $B$.
	The main observation is that all of the $A$-copies in $M$ are actually copies
	from $\mathcal{A}$, so $M$ contains exactly $q$ $A$-copies.
	
	To see this, suppose that the rows of an $A$-copy in $M$ are $i$ and $j + n/2$
	for some $1 \leq i,j \leq n/2$, then there exist $x_1, x_2, x_3, x_4 \in X$ such
	that the entries of the copy were taken from locations $(i,i+x_1), (i, m/2 + i +
	3x_2), (m/2 + j, j-x_3), (m/2+j, m/2+j+x_4)$ in $M$ and so we have $i+x_1 =
	j-x_3$ and $i+3x_2 = j + x_4$. Reordering these two equations we get that $3x_2
	= x_1 + x_3 + x_4$, implying that $x_1 = x_2 = x_3 = x_4$ and $j = i+2x_1$, so
	the above $A$-copy is indeed in $\mathcal{A}$.
	
	Let $n$ be an arbitrarily large positive integer divisible by $m$. Given $M$
	as above, we create an $n \times n$ `blowup' matrix $N$ as follows: For any $1
	\leq i,j \leq n$, $N(i,j) = M(\lfloor im/n \rfloor, \lfloor jm / n \rfloor)$.
	$N$ can also be seen as the result of replacing any entry $e$ in $M$ with an
	$n/m \times n/m$ matrix of entries equal to $e$.
	The total number of $A$-copies in $N$ is exactly $(n/m)^4 q = n^4 |X| / 5m^3$,
	whereas the maximum number of pairwise disjoint $A$-copies in $N$ is exactly
	$(n/m)^2 q = n^2 |X| / 5m$. Assuming that $\epsilon > 0$ is small enough and
	picking $m$ to be the smallest integer divisible by $10$ and larger than
	$\epsilon^{c\log{\epsilon}} $ for a suitable absolute constant $c > 0$
	gives that $|X| / 5m > \epsilon$, but the number of $A$-copies in $N$ is at most
	$n^4 |X| / 5m^3 \leq n^4 / m^2 < \epsilon^{-c\log{\epsilon}}  n^4$ as needed. 
\end{proof}

\section{Concluding remarks}
\label{sec:concluding}
Generally, understanding property testing seems to be easier for objects that are highly
symmetric.
A good example of this phenomenon is the problem of testing properties of
(ordered) one-dimensional binary vectors. There are some results on this
subject,
but it is far from being well understood. On the other hand, the binary vector
properties $P$ that are invariant under permutations of the entries (these are
the properties in which for any vector $v$ that satisfies $P$, any permutation
of the entries of $v$ also satisfies $P$) are merely those that depend only on
the length and the Hamming weight of a vector. This makes the task of testing
these properties trivial.

A central example of the symmetry phenomenon is the well investigated subject of
property testing in (unordered) graphs, that considers only properties of functions from $\binom{[n]}{2}$
to $\{0,1\}$ that are invariant under permutations of $\binom{[n]}{2}$ induced by
permutations on $[n]$. That is, if a labeled graph $G$ satisfies some graph
property, then any relabeling of its vertices results in a graph that also
satisfies this property.
Indeed, the proof of the only known general result on testing properties of \emph{ordered graphs} (here 
the functions are generally not invariant under permutations), given in \cite{AlonBenEliezerFischer2017},
is substantially more complicated than the proof of its unordered analogue.
See \cite{Sudan2010} for further discussion on the role of symmetries in
property testing.

In general, matrices (with row and column order) do not have any symmetries. 
Therefore, the above reasoning suggests that proving results on the testability
of matrix properties is likely to be harder than proving similar results on
properties of matrices where only the rows are ordered (such properties are
invariant under permutations of the columns), which might be harder in turn than
proving the same results for properties of matrices without row and column
orders, i.e.\@ bipartite graphs, as these properties are invariant under permutations of both the rows and the columns.

Theorem \ref{thm:disjoint_copies} is a weak removal lemma for binary matrices with
row and column order, while Theorem \ref{thm:AlonFischerNewman1} is an induced
removal lemma for binary matrices without row and column order, and our
generalization of it, Theorem \ref{thm:removal_row_perm}, is an induced removal lemma
for binary matrices with a row order but without a column order. It will be very
interesting to settle Problem \ref{open:open1}, that asks whether a polynomial induced removal
lemma exists for binary matrices with row and column orders.

It will be interesting to expand our knowledge of matrices in higher dimensions 
and of ordered combinatorial objects in general.
Proposition \ref{obs:obs_non_binary} is a non-induced removal
lemma for (multi-dimensional) matrices without row and column orders. It will be interesting to get
results of this type for less symmetric objects, ultimately for ordered multi-dimensional matrices. 
We believe that providing a direct solution (that does not go through Theorem \ref{thm:ABF_matrix_removal_lemma}) for the following seemingly innocent problem is of interest, and might help providing techniques to help settling Problem \ref{open:open_non_binary} in general. 
In what follows, the \emph{height} of a $2 \times 2$ submatrix $S$ in an $n \times n$ matrix $M$ is the difference between the indices of the rows of $S$ in $M$, divided by $n$.
\begin{open}
	Let $A = \left( 
	\begin{array}{cc}
	0 & 1 \\
	2 & 3
	\end{array}
	\right)$
	and suppose that an $n \times n$ matrix contains $\epsilon n^2$ pairwise disjoint copies of
	$A$. Show (without relying on Theorem \ref{thm:ABF_matrix_removal_lemma}) that
	there exists $\delta = \delta(\epsilon)$ such that $M$ contains an $A$-copy with height at least $\delta$.
\end{open}
The three dimensional analogue of this problem is obviously also of interest. Here Theorem \ref{thm:ABF_matrix_removal_lemma} cannot be applied, so currently we do not know whether such a $\delta = \delta(\epsilon)$ that depends only on $\epsilon$ exists. Solving the three-dimensional analogue will settle Problem \ref{open:open_non_binary} when the forbidden hypermatrix has dimensions $2 \times 2 \times 2$, and the techniques might lead to settling Problem \ref{open:open_non_binary} in its most general form.

As a final remark, in the results in which $\delta^{-1}$ is polynomial in
$\epsilon^{-1}$ we have not tried to obtain tight bounds on the dependence, and
it may be interesting to do so.

\end{document}